\documentclass{amsproc}
\usepackage{amsmath,amssymb,amscd,verbatim}
\usepackage[colorlinks]{hyperref}
\usepackage{bbm}
\usepackage{color}

\let\Bbb\mathbb

\numberwithin{equation}{section}

\theoremstyle{plain}
\newtheorem{theorem}{Theorem}[section]
\newtheorem{lemma}[theorem]{Lemma}
\newtheorem{proposition}[theorem]{Proposition}
\newtheorem{corollary}[theorem]{Corollary}

\theoremstyle{definition}
\newtheorem{definition}[theorem]{Definition}
\newtheorem{example}[theorem]{Example}

\newtheorem{remark}[theorem]{Remark}
\newtheorem{question}[theorem]{Question}

\newtheorem{thmintro}{Theorem}

\DeclareMathOperator\Aut{Aut}

\DeclareMathOperator\Ext{Ext} 
\DeclareMathOperator\GKdim{GKdim}
\DeclareMathOperator\gldim{gldim} 
\DeclareMathOperator\gr{gr}
\DeclareMathOperator\hdet{hdet} 
\DeclareMathOperator\id{id}
\DeclareMathOperator\lcm{lcm}
\DeclareMathOperator\Mod{Mod}
\DeclareMathOperator\op{op}
\DeclareMathOperator\ord{ord}
\DeclareMathOperator\Oz{Oz}
\DeclareMathOperator\supp{supp}
\DeclareMathOperator\tr{tr}
\DeclareMathOperator\Tr{Tr}

\newcommand{\Autgr}{\Aut_{\operatorname{gr}}}

\newcommand\inv{^{-1}}
\newcommand\iso{\cong}
\newcommand\kk{\mathbbm{k}}

\newcommand\bp{\mathbf p}
\newcommand\bq{\mathbf q}
\newcommand\bx{\mathbf x}
\newcommand\bw{\mathbf w}

\newcommand\ou{\overline u}
\newcommand\ov{\overline v}
\newcommand\ow{\overline w}
\newcommand\ox{\overline x}
\newcommand\oy{\overline y} 
\newcommand\oz{\overline z}

\newcommand\NN{\mathbb N}
\newcommand\ZZ{\mathbb Z}

\newcommand\cF{\mathcal F}

\newcommand\degf{d}
\newcommand\ordq{n}

\renewcommand\mod{~\mathrm{mod}~}
\newcommand\restrict[1]{\raisebox{-.3ex}{$|$}_{#1}}

\begin{document}

\title{A family of algebras with trivial ozone group}

\author[Gaddis]{Jason Gaddis}
\address{Miami University, Department of Mathematics, Oxford, Ohio 45056} 
\email{gaddisj@miamioh.edu}

\author[Yee]{Daniel Yee}
\address{University of Wisconsin-Platteville, Mathematics Department, Platteville, WI 53818}
\email{yeed@uwplatt.edu}


\begin{abstract}
We study a family of Calabi--Yau algebras that include the quadratic Artin--Schelter regular algebras associated to a nodal cubic. It is shown that these algebras have trivial ozone group, that is, the identity is the only automorphism that fixes the center pointwise. The graded members of this family of algebras are shown to be rigid in the sense that the invariant ring under a nontrivial group of graded automorphisms is not Artin--Schelter regular.
\end{abstract}

\maketitle

\section{Introduction}
Throughout, $\kk$ is an algebraically closed field of characteristic zero. In this paper, we study the following family of algebras:

\begin{definition}\label{defn.Bq}
For $q \in \kk^\times$ and $f \in \kk[t]$, define
\begin{align}\label{eq.Bq}
B_q(f) = \kk\langle u,v,w \mid 
uv-qvu,
wu-quw-f(v),
wv-q\inv vw - f(u) \rangle.
\end{align}
\end{definition}
\noindent Several subfamilies of the $B_q(f)$ are well-known.
\begin{itemize}
    \item The algebra $B_1(1)$ is isomorphic to $A_1(\kk)[u]$, a polynomial extension of the first Weyl algebra. See Example \ref{ex.weyl}.

    \item The algebra $B_1(t)$ is the enveloping algebra of a three-dimensional solvable Lie algebra. 

    \item The algebras $B_q(0)$ are skew polynomial rings (aka, quantum affine spaces). It is well-known that these algebras are Artin--Schelter regular of global dimension three and they are Calabi--Yau.

    \item The quadratic algebras $B_q(t^2)$ are also Artin--Schelter regular of global dimension three. In this case, the point scheme corresponding to $B_q(t^2)$ is a nodal cubic curve \cite{NVZ}.
\end{itemize}

We are wondering if the entire family of algebras $B_q(f)$ have similar properties to the known algebras above. One of our first results is to prove the following.

\begin{thmintro}[Theorems \ref{thm.cy} and \ref{thm.gldim}]
\label{thm.intro-cy}
The algebra $B_q(f)$ is Calabi--Yau for any $q \in \kk^\times$ and any $f \in \kk[t]$. If $q=1$ or $q$ is a nonroot of unity, then $\gldim  B_q(1)=2$. For all other choices of $q$ and $f$, $\gldim B_q(f)=3$.
\end{thmintro}

When $q$ is a primitive $\ordq$-root of unity, $\ordq \neq 1,3$, then the $B_q(t^2)$ are finite modules over their center, so they satisfy a polynomial identity, or are PI. These PI algebras also appear in Itaba and Mori's classification of quantum projective planes finite over their centers \cite{IM}. It was noted in \cite{CGWZ3} that the PI $B_q(t^2)$ possesses a property stronger than that of being Calabi--Yau and this is the driving idea behind the current investigation.

\begin{definition}\label{defn.ozone}
The \emph{ozone group} of an algebra $A$, denoted $\Oz(A)$, is the subgroup of automorphisms $\phi$ of $A$ such that $\phi(z)=z$ for all $z \in Z(A)$.
\end{definition}

The ozone group necessarily contains the Nakayama automorphism of $A$, hence $\Oz(A)=\{\id_A\}$ implies that $A$ is Calabi--Yau, but the converse is not true. If $q$ is a primitive $\ordq$-root of unity, $q \neq 1$, such that $3 \nmid \ordq$, then $B_q(t^2)$ has trivial ozone group. The only other three-dimensional quadratic Artin--Schelter regular algebras satisfying this condition are certain PI Sklyanin algebras (see \cite[Theorem 0.6]{CGWZ3}). Here we prove this more generally for the $B_q(f)$ satisfying certain conditions.

\begin{thmintro}[Theorem \ref{thm.trivial}]
\label{thm.intro-ozone}
Suppose $f \neq 0$ and let $q$ be a primitive $\ordq$-root of unity such that $n>1$ and $\ordq\nmid(j+1)$ for all $j\in\supp(f)$. Then $\Oz(B_q(f))$ is trivial.
\end{thmintro}

As $\NN$-graded algebras with $d\geq2$, $B_q(t^d)$ are Artin--Schelter regular. We show that they are rigid in the sense that no invariant ring is again Artin--Schelter regular. Under special conditions, the invariant ring may be Artin--Schelter Gorenstein.

\begin{thmintro}[Theorem \ref{thm.gorenstein}]
\label{thm.intro-invariant}
Fix $q \neq \pm 1$ and $d \geq 2$. Set $B=B_q(t^d)$. 
\begin{enumerate}
    \item If $H$ is a nontrivial finite subgroup of the graded automorphism group $\Autgr(B)$, then $B^H$ is not Artin--Schelter regular.
    \item If $H$ is a finite subgroup of $\{ \phi_{a,\xi,h} : a^{d+1}=1\}$, then $B^H$ is Artin--Schelter Gorenstein. (See \eqref{eq.phi_aut} for the definition of $\phi_{a,\xi,h}$.)
\end{enumerate}
\end{thmintro}

\subsection*{Structure of the paper}

In Section \ref{sec.back}, we provide background on common constructions used throughout the paper. We show that $B_q(f)$ is an iterated skew polynomial ring and has either an $\NN$-grading or $\NN$-filtration. It is established that the $B_q(f)$ have Gelfand--Kirillov dimension 3 and we prove Theorem \ref{thm.intro-cy}.

In Section \ref{sec.center} we study the center of the $B_q(f)$. We do not fully compute the center in all cases, but determine enough central elements in order to prove Theorem \ref{thm.intro-ozone}.

In Section \ref{sec.invariant} we compute the graded automorphism group of the $B_q(t^d)$, $d \geq 2$. We determine all finite-order graded automorphisms when $q \neq 1$ and use invariant-theoretic results to prove Theorem \ref{thm.intro-invariant}.

We end in Section \ref{sec.questions} with some additional questions and comments for further study.

\subsection*{Acknowledgements}
Gaddis is partially supported by an AMS–Simons Research Enhancement Grant for PUI Faculty. 

\section{Initial results}
\label{sec.back}

First we establish some isomorphisms which will simplify our presentation.

\begin{lemma}
Fix $q \in \kk^\times$ and $f \in \kk[t]$.
\begin{enumerate}
\item\label{lem.iso1} $B_q(f) \iso B_{q\inv}(f)$.
\item If $c$ is the leading coefficient of $f$, then $B_q(f) \iso B_q(c\inv f)$. That is, up to isomorphism, $f$ is monic.
\item\label{lem.iso2} Let $q=1$, and $f \in \kk[t]$ and $\alpha \in \kk$ so that $g = f(t-\alpha)$. Then $B_1(f) \iso B_1(g)$.
\end{enumerate}
\end{lemma}
\begin{proof}
\eqref{lem.iso1}
The first isomorphism is obtained from the linear map $u \mapsto v$, $v \mapsto u$, and $w \mapsto w$. The second is obtained from the linear map $u \mapsto u$, $v \mapsto v$, and $w \mapsto cw$.
\par\eqref{lem.iso2}
Define a map $\phi:B_1(f) \to B_1(g)$ by $\phi(u)=u-\alpha$, $\phi(v)=v-\alpha$, and $\phi(w)=w$. It is clear that $\phi$ preserves the relation $uv-vu$. Now we have
\[
\phi(wu-uw-f(v)) = w(u-\alpha) - (u-\alpha)w - f(v-\alpha)
    = wu-uw-g(v) = 0.\]
The other relation is checked similarly.
\end{proof}

\subsection{Ore extensions}
Let $A$ be an algebra and $\sigma$ an automorphism of $A$. A \emph{$\sigma$-derivation} of $A$ is a $\kk$-linear map $\delta$ such that $\delta(aa')=\sigma(a)\delta(a')+\delta(a)a'$ for all $a,a' \in A$. Note that $\delta$ is an ordinary derivation exactly when $\sigma=\id_A$. 

Given an algebra $A$, an automorphism $\sigma$ of $A$, and $\sigma$-derivation $\delta$ of $A$, the \emph{Ore extension} $C=A[x;\sigma,\delta]$ is generated as an algebra over $A$ by the indeterminate $x$ subject to the relations $xa = \sigma(a)x+\delta(a)$ for all $a \in A$.

For $q \in \kk^\times$, the corresponding \emph{quantum plane} is defined as 
\[\kk_q[u,v]=\kk\langle u,v \mid uv-qvu\rangle.\]

\begin{lemma}\label{lem.ore}
For $f \in \kk[t]$, $B_q(f)$ is the 
Ore extension $\kk_q[u,v][w;\sigma,\delta]$ where
\[ 
\sigma(u)=qu, \quad \sigma(v)=q\inv v, \qquad 
\delta(u)=f(v), \quad \delta(v)=f(u).\]
\end{lemma}
\begin{proof}
We need only verify that $\delta$ is a $\sigma$-derivation of $\kk_q[u,v]$, which we do below:
\begin{align*}
\delta(uv-qvu)
    &= (\sigma(u)\delta(v) + \delta(u)v) 
        - q(\sigma(v)\delta(u) + \delta(v)u) \\
    &= (qu f(u) + f(v)v) - q(q\inv v f(v) + f(u)u) = 0.
\end{align*}
This proves the result.
\end{proof}

It follows from the Ore extension presentation that $\{ v^i u^j w^k : i,j,k \in \NN\}$ is a $\kk$-algebra basis for $B_q(f)$.

\subsection{Gradings, filtrations, and regularity}
\label{sec.gradings}

Let $\Gamma$ be an abelian monoid. An algebra $A$ is \emph{$\Gamma$-graded} if there is a vector space decomposition $A = \bigoplus_{\gamma \in \Gamma} A_\gamma$ such that $A_\gamma A_{\gamma'} \subset A_{\gamma + \gamma'}$. 
For any $\degf \in \NN$, $B_q(t^\degf)$ is $\ZZ$-graded by setting
$\deg u = \deg v = 1$ and $\deg w = \degf-1$. 

If $\degf \geq 2$, then $B_q(t^\degf)$ is $\NN$-graded and \emph{connected} (i.e., $(B_q(t^\degf))_0 = \kk$). In particular, $B_q(t^\degf)$ is a \emph{quantum plane} of weight $(1,1,\degf)$ \cite{steph2}. In the special case $f=0$, we will prefer to use the grading $\deg u = \deg v = \deg w = 1$, so $B_q(0)$ is also connected $\NN$-graded.

An $\NN$-graded algebra $A$ is \emph{Artin--Schelter (AS) Gorenstein} of dimension $d$ if $A$ has injective dimension $d<\infty$ on the left and on the right, and there are isomorphisms:
$\Ext_A^i({}_A\kk,A) \iso \delta_{id} \kk_A$
where $\delta_{id}$ is the Kronecker-delta.
If $A$ additionally has finite global dimension and Gelfand-Kirillov (GK) dimension, then $A$ is 
\emph{Artin--Schelter (AS) regular}. The following is well-known.

\begin{lemma}\label{lem.asreg}
(1) The algebra $B_q(0)$ is AS regular of global and GK dimension three with Hilbert series $(1-t)^{-3}$.

(2) For $\degf \geq 2$, the algebra $B_q(t^\degf)$ is AS regular of global and GK dimension three with Hilbert series $(1-t)^{-2}(1-t^\degf)\inv$.
\end{lemma}
\begin{proof}
(1) See \cite{ASch}.

(2) See \cite{steph}. Alternatively, it follows from Lemma \ref{lem.ore} and \cite[Proposition 2]{ASTgln}.
\end{proof}

An \emph{$\NN$-filteration} $\cF$ on an algebra $A$ is a collection of vector subspaces $\{F_k A\}$ of $A$ satisfying $A = \bigcup_{k \in \NN} F_k A$, $F_k A \subset F_{k+1} A$ and $F_k A F_\ell \subset F_{k+\ell} A$. Given a filtration $\cF$ of $A$, the \emph{associated graded ring} is $\gr_\cF A = \bigoplus_{k \in \NN} F_k A/F_{k-1} A$ (where $F_{-1}=0$). We drop the subscript $\cF$ when the filtration is understood.

If $f$ is not a monomial but $\deg(f) = \degf \geq 2$, then again setting $\deg u = \deg v = 1$ and $\deg w = \degf-1$ defines a filtration such that the associated graded ring $\gr B_q(f)$ is isomorphic to $B_q(t^\degf)$. If $\degf < 2$, then setting $\deg u = \deg v = \deg w = 1$ gives $\gr B_q(f) \iso B_q(0)$.

\begin{lemma}\label{lem.GK}
For any $q \in \kk^\times$ and any $f \in \kk[t]$, $\GKdim B_q(f) = 3$.
\end{lemma}
\begin{proof}
Since $B_q(f)$ is affine, then $\GKdim B_q(f) = \GKdim \gr B_q(f)$. The result now follows from the Lemma \ref{lem.asreg} and the discussion above.
\end{proof}

\subsection{The Calabi--Yau property}
We denote the enveloping algebra $A \otimes A^{\op}$ of an algebra $A$ by $A^e$. An algebra $A$ is \emph{homologically smooth} if there exists a finite resolution of $A$ in $A^e-\Mod$ consisting of finite length projectives. Then $A$ is \emph{Calabi--Yau} (CY) of dimension $d$ if it is homologically smooth and there exist isomorphisms
\[ \Ext_{A^e}^i(A,A^e) \iso \delta_{id} A.\]

The notion of a CY algebra was introduced by Ginzburg \cite{ginz}. Examples of CY algebras include polynomial rings over a field, the Weyl algebra, and Sklyanin algebras. The algebras $B_q(t^2)$ are also known to be CY. We show below that the $B_q(f)$ are all CY.

Following Ginzburg's philosophy, every CY algebra appearing ``in the wild" should come from a potential. Let $\bw \in \kk\langle u,v,w\rangle$ be homogeneous of degree $N$. We say $\bw$ is a \emph{potential} if $\bw$ is closed under the map \[ a_1a_2\cdots a_N \mapsto a_Na_1a_2 \cdots a_{N-1}\] for $a_i \in \{u,v,w\}$. Then we define 
\[ \partial_b(a_1a_2\cdots a_N) = 
\begin{cases}
    b\inv a_1a_2\cdots a_N = a_2\cdots a_N & \text{if $b=a_1$} \\
    0 & \text{otherwise},
\end{cases}\]
extended linearly. We denote $\bw_b = \partial_b(\bw)$. 

More generally, a nonhomogeneous potential $\bw$ is a sum of homogeneous potentials. The algebra $\kk\langle u,v,w\rangle/(\bw_b : b \in \{u,v,w\})$ is called a \emph{derivation-quotient algebra}.

Suppose $\deg f = \degf$ and write
\begin{align}\label{eq.fform} 
    f(t) = \sum_{j=0}^\degf c_j t^j.
\end{align}
Set
\[ F(t) = \sum_{j=0}^\degf c_j t^{j+1}\] 
and define the potential 
\begin{align}\label{eq.potential}
	\bw=(uvw+vwu+wuv) - q(vuw+uwv+wvu) + qF(u) - F(v).
\end{align} 
Then
\[
	\bw_u = vw - qwv + qf(u), \qquad \bw_v = wu - quw - f(v), \qquad \bw_w = uv - qvu.
\]
Factoring $-q$ out of the first relation gives $\bw_u' = wv-q\inv vw - f(u)$. Then $\bw_u', \bw_v, \bw_w$ are exactly the relations defining the $B_q$.

\begin{lemma}\label{lem.cy1}
Suppose $\degf \leq 1$. Then $B_q(f)$ is CY.
\end{lemma}
\begin{proof}
If $f=0$ (so $F=0$ as well), then $\bw$ is the potential corresponding to a CY skew polynomial ring \cite[Theorem 0.10]{CGWZ2}. The result now follows from \cite[Theorem 1.1]{BT}.
\end{proof}

We now turn to the case $\degf \geq 2$. 
Let $\bx^T=(u ~ v ~ w)$ and set $B=B_q(t^\degf)$ for some $\degf \geq 2$. The projective resolution of the trivial module ${}_B\kk = B_q(t^\degf)/(u,v,w)$ is 
\[
0 \to B[-\degf-1] \xrightarrow[]{\bx^T} B[-\degf]^{\oplus 2} \oplus B[-2] \xrightarrow[]{M} B[-1]^{\oplus 2} B[-\degf+1] \xrightarrow[]{\bx} B \to {}_A \kk \to 0
\]
where
\[ M = \begin{pmatrix}qu^{\degf-1} & -qw & v \\ w & -v^{\degf-1} & -qu \\ -qv & u & 0\end{pmatrix}.\]

\begin{theorem}\label{thm.cy}
The algebra $B_q(f)$ is CY for any $q \in \kk^\times$ and any $f \in \kk[t]$.
\end{theorem}
\begin{proof}
The result for $\degf < 2$ follows from Lemma \ref{lem.cy1}. Assume $\degf \geq 2$. Then $B_q(t^\degf)$ is AS regular by Lemma \ref{lem.asreg}, hence it is twisted CY by \cite[Lemma 1.2]{RRZ}.
Since $B_q(t^\degf)$ is a derivation-quotient algebra, then $B_q(t^\degf)$ is CY. The general case now follows from the definition of $B_q(f)$ as a derivation-quotient algebra and \cite[Theorem 1.1]{BT}.
\end{proof}

In contrast to GK dimension above, the question of global dimension is more sensitive. By \cite[Theorem 5.3 (i)]{MR}, $2 \leq \gldim B_q(f) \leq 3$. It is well-known that $B_q(t^d)$ has global dimension 3 for $d \geq 2$ (see Lemma \ref{lem.asreg}). This is also true of $B_q(0)$, which is a skew polynomial ring, and $B_1(t)$, which is isomorphic to the enveloping algebra of a three-dimensional Lie algebra. On the other hand, we have the following example.

\begin{example}\label{ex.weyl}
Let $B=B_1(1)$, which has presentation is
\[ B = \kk\langle u,v,w \mid uv-vu, wu-uw-1, wv-vw-1\rangle.\]
Set $\tilde{u}=u-v$, then $\tilde{u},v,w$ is a generating space for $B$ with $\tilde{u}$ central. It is then clear that $B \iso A_1(\kk)[\tilde{u}]$, so $\gldim B = 2$.
\end{example}

The next result shows that, generically, $B_q(f)$ has global dimension 3.

\begin{lemma}\label{lem.gldim}
Suppose that either 
\begin{enumerate}
    \item \label{gldim3} $q \neq 1$ and $f(\alpha)=0$ for some $\alpha \neq 0$,
    \item \label{gldim1} $f=t^d$ for $d \geq 1$,
    \item \label{gldim2} $q=1$ and $d \geq 1$, or
    \item \label{gldim4} $f(t)=1$ and $q$ is a primitive $\ordq$-root of unity, $\ordq>1$.
\end{enumerate}
Then $\gldim B_q(f) = 3$.
\end{lemma}
\begin{proof}
By \cite[Remark 2.8]{BT}, it suffices in each case to demonstrate that $B_q(f)$ has a nontrivial finite-dimensional module.

\eqref{gldim3} 
Set $c=c_0=f(0)$. Let $M=\kk m$ and give $M$ a $B$-module structure by $u.m=0$, $v.m = \alpha m$, and $w.m = \frac{c}{(1-q\inv)\alpha}m$. It is easy to verify that this structure is well defined.

\eqref{gldim1} 
Let $M=\kk m$ and define a module structure by $u.m = v.m = w.m = 0$.

\eqref{gldim2} 
By applying the isomorphism in Lemma \ref{lem.iso2}, we may assume that $c_0=0$. We now apply \eqref{gldim1}.

\eqref{gldim4}
Set $B=B_q(1)$. It is reasonably easy to see that $u^n,v^n,w^m \in Z(B)$. This also follows from results in Section \ref{sec.center}. Now $M=B/(u^n,v^n,w^n)$ is a finite-dimensional module of $B$.
\end{proof}

\begin{remark}
An alternate way of obtaining the \eqref{gldim2} above is as follows. Recall that $B$ is the Ore extension $\kk[u,v][w;\delta]$. Let $\alpha=f(0)$ and note that $J=(u-v,v-\alpha)$ is a height two prime ideal of $\kk[u,v]$ such that $\delta(J) \subset J$. Now $\gldim B_1(f) = \gldim \kk[u,v] + 1 = 3$ by \cite[Theorem 8]{Gdim1} (see also \cite[Theorem 10.3]{MR}).
\end{remark}

\begin{theorem}\label{thm.gldim}
If $q=1$ or $q$ is a nonroot of unity, then $\gldim  B_q(1)=2$. For all other choices of $q$ and $f$, $\gldim B_q(f)=3$.
\end{theorem}
\begin{proof}
By Example \ref{ex.weyl} and Lemma \ref{lem.gldim}, we may assume that $B=B_q(1)$ with $q$ a nonroot of unity and $q \neq 1$. We refer to the discussion preceding \cite[Proposition 1.4, p.26]{BS}. By that reference, with $a=c=d=0$, $b=1$ and $\alpha=\beta=f=q$, all finite-dimensional simple modules of $B$ are one-dimensional.

Let $M=\kk m$ be a simple module of $B$ so that $u\cdot m = \lambda_u m$, $v\cdot m = \lambda_v m$, and $w\cdot m = \lambda_w m$ for $\lambda_u,\lambda_v,\lambda_w \in \kk$. Since we have assumed $q \neq 1$, then $(uv-qvu)\cdot m = 0$ implies that $\lambda_u=0$ or $\lambda_v=0$. WLOG, assume $\lambda_u=0$. Then $(wu-quw-1)\cdot m=0$ implies that $m=0$. By \cite[Remark 2.8]{BT}, $\gldim B \neq 3$. So, by \cite[Theorem 5.3]{MR}, $\gldim B = 2$.
\end{proof}

\section{\texorpdfstring{The center and ozone group of $B_q(f)$}{The center and ozone group of for Bq(f)}}
\label{sec.center}

In this section we compute the center of $B_q(f)$ for generic $q$, denoted $Z(B_q(f))$, and we show that certain elements are central when $q$ is a root of unity satisfying certain conditions relative to the polynomial $f$. These latter computations will be useful in considering the ozone group of $B_q(f)$ roots of unity.

For a positive integer $k$, define the $q$-number to be
\[ [k]_q = \frac{q^k-1}{q-1} = 1 + q + \cdots + q^{k-1}.\]
Note that $[k]_q=0$ if and only if $q$ is an $\ordq$-root of unity and $k \equiv 0 \mod\ordq$.
More generally, if we set $[k]_q! = [k]_q [k-1]_q \cdots [1]_q$, then the \emph{Gaussian binomial coefficients} are defined as
\[
    \binom{k}{i}_q = \frac{[k]_q!}{[i]_q![k-i]_q!}.
\]

For a polynomial $f$ as in \eqref{eq.fform},
we define the \emph{support} of $f$ to be 
\[\supp(f) = \{i \mid c_i \neq 0\}.\]

First we prove some useful identities.

\begin{lemma}\label{lem.cnt_id1}
For all $k \geq 1$,
\begin{align}
\label{eq.id1}
\delta(u^k) &= \sum_{j=0}^{\degf} [k]_{q^{j+1}} c_j v^j u^{k-1} \\
\label{eq.id2}
\delta(v^k) &= \sum_{j=0}^{\degf} [k]_{q^{-(j+1)}} c_j u^j v^{k-1}.
\end{align}
Consequently,
\begin{align}
\label{eq.comm1}
wu^k &= q^k u^k w + \sum_{j=0}^d [k]_{q^{j+1}} c_j v^j u^{k-1} \\
\label{eq.comm2}
wv^k &= q^{-k}v^kw + \sum_{j=0}^d [k]_{q^{-(j+1)}} c_j u^j v^{k-1}.
\end{align}
\end{lemma}
\begin{proof}
We prove \eqref{eq.id1} and \eqref{eq.id2} is similar. Since $\delta(u)=f(v)$, then this proves the case $k=1$. Suppose the identity holds for some $k \geq 1$. Then
\begin{align*}
\delta(u^{k+1})
    &= \sigma(u)\delta(u^k) + \delta(u)u^k \\
    &= qu \left( \sum_{j=0}^{\degf} [k]_{q^{j+1}} c_j v^j u^{k-1}\right) + \left(\sum_{j=0}^{\degf} c_j v^j\right) c_jv^j u^k \\
    &= \sum_{j=0}^{\degf} \left( q^{j+1}[k]_{q^{j+1}} + 1\right) c_jv^j u^k \\
    &= \sum_{j=0}^{\degf} [k+1]_{q^{j+1}} c_jv^j u^k.
\end{align*}
The result follows from induction.
\end{proof}

\begin{lemma}\label{lem.uv}
Let $q$ be a primitive $\ordq$-root of unity and $\ordq\nmid(j+1)$ for all $j\in\supp(f)$. Then $u^\ordq,v^\ordq\in Z(B_q(f))$.
\end{lemma}
\begin{proof} 
Let $j \in \supp(f)$. Since $\ordq\nmid(j+1)$, then $q^{j+1} \neq 1$ but $(q^{j+1})^n=1$. Hence, $[n]_{q^{j+1}}=0$ and $[n]_{q^{-(j+1)}}=0$. The result now follows by Lemma \ref{lem.cnt_id1}.
\end{proof}

\begin{remark}
The assumption $\ordq \nmid (j+1)$ for all $j\in\supp(f)$ is essential. For example, if $f(t)=t^8$ and $q^3=1$, then the sum in \eqref{eq.comm1} becomes
\[ \sum_{i=0}^{k-1}q^{i(8+1)}v^8u^{k-1}=kv^8u^{k-1}\neq0.\]
However, if $\kk$ is an algebraically closed field of characteristic $p>0$, then the sum becomes zero whenever $p\mid k$.
\end{remark}

\begin{lemma}
Suppose $f=t^d$. Then for $k\geq 1$,
\begin{align}
\label{eq.wku}
    w^k u &= \sum_{i=0}^k \binom{k}{i}_{q^{d+1}} \sigma^{k-i}\delta^i(u) w^{k-i}\\
\label{eq.wkv}
    w^k v &= \sum_{i=0}^k \binom{k}{i}_{q^{-(d+1)}} \sigma^{k-i}\delta^i(v) w^{k-i}.
\end{align}
\end{lemma}
\begin{proof}
We have
\begin{align*}
q^{d+1}\sigma\delta(u) 
    &= q^{d+1}\sigma(v^d) 
    = q^{d+1}(q^{-d} v^d) 
    = qv^d 
    = \delta(qu) 
    = \delta\sigma(u) \\
q^{-(d+1)}\sigma\delta(v)
    &= q^{-(d+1)}\sigma(u^d)  
    = q^{-(d+1)}(q^d u^d)
    = q\inv u^d
    = \delta(q\inv v)
    = \delta(\sigma(v)).
\end{align*}
Since $\delta\sigma(u) = q^{d+1}\sigma\delta(u)$, then when restricting to the action on $u$, we have that $(\sigma,\delta)$ is a $q^{d+1}$-derivation. Similarly, for the action on $v$, $(\sigma,\delta)$ is a $q^{-(d+1)}$-derivation. The result now follows from \cite[6.2]{G}.
\end{proof}

\begin{remark} In the next lemma, we will use a reformulation of the previous lemma, namely for any $k\geq1$, and any polynomial $f(t)$,
\begin{align}\label{rewr1}w^ku=q^kuw^k+\sum_{i=0}^{k-1}w^{k-1-i}f(v)(qw)^i.\end{align}
\end{remark}

\begin{lemma}\label{lem.wpwr}
Assume that $q$ is a primitive $\ordq$-root of unity with $\ordq>1$. Then $f(u),f(v)\in Z(B_q(f))$ if and only if $\ordq\mid j$ for all $j\in\supp(f)$. Consequently, if $\ordq\mid j$ for all $j\in\supp(f)$, then $w^{\ordq}\in Z(B_q(f))$.
\end{lemma}
\begin{proof} 
Write $f$ as in \eqref{eq.fform}.
If $f(u),f(v)\in Z(B_q(f))$, then $f(u)v=vf(u)$. But the relation $uv=qvu$ implies that $f(u)v=vf(qu)$. Since $B_q(f)$ is a domain, $f(u)=f(qu)$, and so for each $c_j\neq0$, linear independence implies $q^j-1=0$, hence $\ordq\mid j$ for all $j\in\supp(f)$.

Conversely, assume that $\ordq\mid j$ for all $j\in\supp(f)$. Then $f(qu)=f(q^{-1}u)=f(u)$ and $f(q^{-1}v)=f(qv)=f(v)$, hence $uf(v)=f(q^{-1}v)u=f(v)u$ and $f(u)v=vf(qu)=vf(u)$. By Lemma \ref{lem.cnt_id1},
\[
\delta(f(u))
    = \sum_{j=0}c_j\delta(u^j)    
    = \sum_{j=0}^dc_j\left(\sum_{i=0}^d[j]_{q^{i+1}}c_iv^i\right)u^{j-1}.
\]
Since $\ordq=\ord(q)$ and $n \mid j$, then $[j]_{q^{i+1}} = 0$ for all $i$. We have shown that $\delta(f(u))=0$ and similarly $\delta(f(v))=0$. Therefore, $wf(u)=f(qu)w+\delta(f(u))=f(u)w$, whence $f(u)\in Z(B_q(f))$ and similarly $f(v)\in Z(B_q(f))$.

We continue to assume that $\ordq \mid j$ for all $j\in\supp(f)$. By the above, this is equivalent to assuming $f(u),f(v)\in Z(B_q(f))$. We claim that $w^{\ordq}$ is central. We apply equation \ref{rewr1} for all $k\geq1$, using the assumption $f(v)\in Z(B_q(f))$, we have that
\begin{align*}
w^{\ordq}u
    &=q^{\ordq}uw^{\ordq}+\sum_{i=0}^{\ordq-1}q^iw^{\ordq-1-i}f(v)w^i
    =uw^\ordq+\left(\sum_{i=0}^{\ordq-1}q^i\right)f(v)w^{\ordq-1}\\
    &=uw^\ordq+[\ordq]_{q}f(v)w^{\ordq-1}
    =uw^\ordq.
\end{align*}
Similarly, using $f(u)\in Z(B_q(f))$ shows that $w^nv = vw^n$, proving our claim.
\end{proof}

\begin{remark}
If we remove the hypothesis $\ordq \mid j$ from Lemma \ref{lem.wpwr}, then it is not true that $w^\ordq$ is always central. In the case that $\ordq=4$ and $f=t^2$, then one checks that $w^4$ is not central, contradicting \cite[Lemma 1.7]{CGWZ3}. Instead, we have
\[ g = w^4 + 2(1-q) \Omega w \in Z(B_q(f)).\]

It remains true that, for $q$ a primitive root of unity of order $\ordq \neq 1,3$, $B_q(t^2)$ is module-finite over its center. Moreover, the ozone computation in \cite{CGWZ3} do not use the centrality of $w^\ordq$. See Proposition \ref{prop.PI} below.
\end{remark}

\par We now establish another canonical central element of $B_q(f)$.

\begin{lemma}\label{lem.inner_deriv}
Let $X$ be the Ore set of $B_q(f)$ generated by powers of $u$ and $v$. Suppose the order of $q$ does not divide $j+1$ for all $j \in \supp(f)$. Then $\delta$ is an inner $\sigma$-derivation on $\widehat{B_q(f)} = (B_q(f))X\inv$.
\end{lemma}
\begin{proof}
There is a unique extension of $\sigma$ and $\delta$ to $(B_q(f))X\inv$, which by an abuse of notation we also refer to as $\sigma$ and $\delta$, respectively. Then $\widehat{B_q(f)} = \kk_q[u^{\pm 1},v^{\pm 1}][w;\sigma,\delta]$. 

Suppose $\deg f = \degf$ and write $f$ as in \eqref{eq.fform}. For each $j \in \supp(f)$, set $\alpha_j = (q^{-j} - q)\inv$ and $\beta_j=(q\inv-q^j)\inv$. By hypothesis, $\alpha_j$ and $\beta_j$ are well-defined for all $j \in \supp(f)$. Define
\[
    \gamma = u\inv \left( \sum_{j=0}^{\degf} \alpha_j c_j v^j \right)
        - v\inv \left( \sum_{j=0}^{\degf} \beta_j c_j u^j \right).
\]
We have
\begin{align*}
\gamma u-\sigma(u)\gamma 
    &= u\inv \left( \sum_{j=0}^{\degf} \alpha_j c_j v^j \right)u
        - v\inv \left( \sum_{j=0}^{\degf} \beta_j c_j u^j \right)u \\
    &\hspace{5em}
        - q \left( \sum_{j=0}^{\degf} \alpha_j c_j v^j \right)
        + qu v\inv \left( \sum_{i=0}^{\degf} \beta_j c_j u^j \right) \\
    &= \left( \sum_{j=0}^{\degf} \alpha_j q^{-j} c_j v^j \right)
        - v\inv u \left( \sum_{j=0}^{\degf} \beta_j c_j u^j \right) \\
   &\hspace{5em}    
        - q \left( \sum_{j=0}^{\degf} \alpha_j c_j v^j \right)
        + v\inv u \left( \sum_{j=0}^{\degf} \beta_j c_j u^j \right) \\
    &= \sum_{i=0}^{\degf} \alpha_j (q^{-j}-q) c_j v^j
    = f(v).
\end{align*}
The proof that $\gamma v-\sigma(v)\gamma=f(u)$ is similar.
\end{proof}

In $\widehat{B_q(f)}$, set $w'=w-\gamma$ and 
\begin{align}\label{eq.Omega}
\Omega = uvw' = uvw + \left( \sum_{j=0}^{\degf} \beta_j c_j u^{j+1} \right)
    - q\left( \sum_{j=0}^{\degf} \alpha_j c_j v^{j+1} \right)
\in B_q(f)
\end{align}

\begin{proposition}\label{prop.cnt}
The element $\Omega$ is central in $B_q(f)$. 
\begin{enumerate}
    \item If $q=1$, then $Z(B_1(f))=\kk[uf(u)-vf(v)]$.
    \item If $q$ is not a root of unity, then $Z(B_q(f))=\kk[\Omega]$.
    \item If $q$ is a primitive $\ordq$-root of unity and $\ordq\nmid(j+1)$ for all $j\in\supp(f)$, then $u^\ordq,v^\ordq,\Omega\in Z(B_q(f))$.
\end{enumerate}
\end{proposition}
\begin{proof}
Suppose $q=1$. Then it is well-known that $Z(B_1(f)) = \kk[u,v]^\delta$. By \cite[Proposition 6.9.3]{nowicki}, since $f(u)$ and $f(v)$ are coprime, then $\kk[u,v]^\delta = \kk[uf(u)-vf(v)]$. This proves (1).

By Lemma \ref{lem.inner_deriv},
\[ \widehat{B_q(f)} = \kk_q[u^{\pm 1},v^{\pm 1}][w';\sigma].\]
Let $S=\kk_\bp[u_1,u_2,u_3]$ be the skew polynomial ring with 
parameters $p_{12}=q\inv$, $p_{13} = q$, $p_{23}=q\inv$. Let 
$Y$ be the Ore set of $S$ generated by powers of $u_1$ and $u_2$, 
and set $\widehat{S}=SY\inv$. Then it is clear that 
$\widehat{S} \iso \widehat{B_q(f)}$ via the isomorphism
$u_1 \mapsto u$, $u_2 \mapsto v$, and $u_3 \mapsto w'$.

If $q$ is not a root of unity, then it is easy to check that $Z(S)=Z(\widehat{S}) =\kk[\Omega]$. Since $Z(B_q(f)) = Z(\widehat{B_q(f)}) \cap B_q(f)$. Clearing fractions gives (2).

Now suppose that $q$ is a primitive $\ordq$-root of unity and $\ordq\nmid(j+1)$ for all $j\in\supp(f)$. Again, since $S$ is CY, then $Z(S)$ is generated by $u_1^n$, $u_2^n$, $u_3^n$, and $u_1u_2u_3$ \cite[Example 5.2(1)]{CGWZ2}. Hence, $Z(\widehat{S})$ is generated by the same elements along with appropriate inverses. Again using $Z(B_q(f)) = Z(\widehat{B_q(f)}) \cap B_q(f)$, it follows that $u^n$, $v^n$, and $uvw'=\Omega$ are all central in $B_q(f)$.
\end{proof}

The next result shows that, under our standard hypotheses, $B_q(f)$ is module-finite over its center.

\begin{proposition}\label{prop.PI}
Let $q$ be a primitive $\ordq$-root of unity and $\ordq\nmid(j+1)$ for all $j\in\supp(f)$. Then 
\begin{enumerate}
    \item $B_q(f)$ is PI, and
    \item $B_q(f)$ is module-finite over its affine center.
\end{enumerate}
\end{proposition}
\begin{proof}
(1) Set $B=B_q(f)$ and $Z=Z(B)$. Now we prove that, under the hypotheses of the lemma, that $B$ is PI. Since $B$ is a domain, then we may consider the localization $\widehat{B}=BZ\inv$.

Suppose $B$ is not PI. Since $B$ is affine, this implies that it is also not locally PI (every finitely generated subalgebra is PI). Thus, $\widehat{B}$ is not locally PI (if it were, then the finitely generated subalgebra $B$ would be locally PI).
Now, \cite[Corollary 2]{SmZ} gives $\GKdim B \geq 2 + \GKdim Z$. By Lemma \ref{lem.GK}, $\GKdim B=3$. Since $\kk[u^n,v^n] \subset Z$ by the above, then $\GKdim Z \geq 2$, a contradiction. Thus, $B$ is PI. 

(2) Let $\degf=\deg(f)$. Set a filtration on $B$ by $\deg(u)=\deg(v)=1$ and $\deg(w)=d$ so that the associated graded ring is isomorphic to $A=\kk_q[u,v][w;\sigma]$. Then $A$ is PI and by \cite[Lemma 1]{VV}, $A$ is a gr-maximal order in a gr-simple gr-Artinian ring. Thus, by \cite[Theorem 5]{VV}, $B$ is a maximal order. Since $B$ is PI by (1), then $B$ is equal to its trace ring by \cite[Proposition 9.8]{MR}. Finally, since $B$ is affine and noetherian, then \cite[Proposition 9.11]{MR} implies $R$ is module-finite over its affine center.
\end{proof}

If $q$ is not a root of unity, then $B_q(f)$ is not PI. If it were, then the subring $\kk_q[u,v]$ would be PI, but a quantum plane is PI if and only if $q$ is a root of unity.

\subsection{The ozone group of $B_q(f)$}

Using Proposition \ref{prop.cnt}, and under the same hypotheses as Proposition \ref{prop.cnt}~(3), we now show that $B_q(f)$ has trivial ozone group. For $n=2$ and $q$ a root of unity of order not $1$ or $3$, this is done in \cite{CGWZ3}. 

Let $A$ be a prime algebra which is a finite module over its center $Z$ and $\phi \in \Oz(A)$. By \cite[Lemma 1.9]{CGWZ3}, there is a regular normal element $a \in A$ such that $\phi=\eta_a$, where $\eta_a(b) = a\inv b a$. 
If $A$ is a $G$-graded domain for some finite abelian group $G$, then we may further assume that $b$ is homogeneous and, moreover, if $G=\ZZ^n$ then we have that $\Oz_{\gr}(A)=\Oz(A)$.

\begin{example}
The ozone group $\Oz(B_q(0))$ is generated by $\eta_u$, $\eta_v$, and $\eta_w$. See \cite{CGWZ2,CGWZ3} for a more thorough analysis of the ozone groups in this setting. 
\end{example}

Suppose that $A$ is now merely a $G$-filtered domain. 
Since $a\eta_a(b) = ba$, then the domain hypothesis shows that $\deg b = \deg \eta_a(b) = \deg\phi(b)$.

\begin{lemma}\label{lem.oz1}
Suppose $f \neq 0$ and let $q$ be a primitive $\ordq$-root of unity such that $\ordq>1$ and $\ordq\nmid(j+1)$ for all $j\in\supp(f)$. If $\phi \in \Oz(B_q(f))$, then $\phi(u),\phi(v) \in \kk_q[u,v]$.
\end{lemma}
\begin{proof}
Let $\phi \in \Oz(B_q(f))$. Then $u^\ordq \in Z(B_q(f))$ and so $u^n=\phi(u^n)$. Since $B_q(f)$ is a domain, then a degree argument shows that $\deg\phi(u) = 1$. Similarly, $\deg\phi(v)=1$. We will be using the appropriate grading as considered in \S \ref{sec.gradings}.

Write 
\[ 
\phi(u)=a_{11}u+a_{12}v+a_{13}w + b_1 
    \quad\text{and}\quad
\phi(v)=a_{21}u+a_{22}v+a_{23}w + b_2\] 
for $a_{ij},b_i \in \kk$. Note if $d>2$, then $a_{13}=a_{23}=0$ and we are done. Assume now that $d=2$. Then the coefficient of $w^2$ in $\phi(uv-qvu)=0$ is $a_{13}a_{23}=0$. 

Suppose $a_{13}=0$. The argument assuming that $a_{23}=0$ is similar. The coefficient of $w$ in $\phi(uv-qvu)$ is $b_1a_{23}$. If $a_{23}=0$ then we are done, so assume that $b_1=0$. Now expanding $\phi(u)^n$ shows that $a_{12}=0$, so $\phi(u)=a_{11} u$ where $a_{11}^n=1$. Then
\begin{align*}
0   &= \phi(uv-qvu) \\
    &= (a_{11}u)(a_{21}u+a_{22}v+a_{23}w + b_2) \\
        &\qquad -q(a_{21}u+a_{22}v+a_{23}w + b_2)(a_{11}u) \\
    &= (1-q)(a_{11}a_{21}u^2 + b_2a_{11} u) + a_{11}a_{23}(uw-qwu) \\
    &= (1-q)(a_{11}a_{21}u^2 + b_2a_{11} u) + a_{11}a_{23}f(v).
\end{align*}
Since $f \neq 0$, then it follows that $a_{23}=0$.
\end{proof}

\begin{theorem}\label{thm.trivial}
Suppose $f \neq 0$ and let $q$ be a primitive $\ordq$-root of unity such that $n>1$ and $\ordq\nmid(j+1)$ for all $j\in\supp(f)$. Then $\Oz(B_q(f))$ is trivial.
\end{theorem}
\begin{proof} 
Let $\phi \in \Oz(B_q(f))$. By Lemma \ref{lem.oz1}, $\phi(u),\phi(v) \in \kk_q[u,v]$. Since $q \neq 1$, then all automorphisms of $\kk_q[u,v]$ are graded. As $\phi(u^n)=u^n$ and $\phi(v^n)=v^n$, then $\phi(u) = \xi u$ and $\phi(v)=\xi' v$ for some $\ordq$-root of unity $\xi$.

Suppose $d \geq 2$. Since $\Omega$ is central, then $\phi(\Omega)=\Omega$. 
Thus, $\deg(\phi(w)) = d-1$. Write $\phi(w) = aw + h$ for some $a \in \kk^\times$ and some $h \in \kk_q[u,v]_{(d-1)}$.
Then since $\Omega \in Z(B_q(f)$, we have
\begin{align*}
\Omega &= \phi(\Omega)
    = (\xi u)(\xi' v)(aw + h) + \left( \sum_{j=0}^{\degf} \beta_j c_j (\xi u)^{j+1} \right)
    - q\left( \sum_{j=0}^{\degf} \alpha_j c_j (\xi' v)^{j+1} \right) \\
    &= \xi\xi' a (uvw) + \xi\xi' uv h 
    + \left( \sum_{j=0}^{\degf} \xi^{j+1}\beta_j c_j u^{j+1} \right)
    - q\left( \sum_{j=0}^{\degf} \alpha_j c_j (\xi')^{j+1} v^{j+1} \right).
\end{align*}
Since $\deg_u(uvh) \geq 1$ and $\deg_v(uvh) \geq 1$, then it follows that $h=0$. Furthermore, since $\ordq \nmid d+1$, it follows that $\xi=\xi'=1$. But then $a=1$, so $\phi$ is trivial.

The case $d<2$ is similar. In this case, $\deg u = \deg v = \deg w = 1$ and again since $\phi(\Omega)=\Omega$, we have $\deg(\phi(w))=1$. Thus, we may write $\phi(w) = aw + h$ for some $a \in \kk^\times$ and some $h \in \kk_q[u,v]_1$. The remainder of the argument is identical. 
\end{proof}

The next corollary follows directly from \cite[Theorem A]{CGWZ3}. 

\begin{corollary}
Suppose $f \neq 0$ and let $q$ be a primitive $\ordq$-root of unity such that $n>1$ and $\ordq\nmid(j+1)$ for all $j\in\supp(f)$. Then every normal element of $B_q(f)$ is central.
\end{corollary}

\section{Invariant Theory}
\label{sec.invariant}

We consider automorphisms of the $B_q(f)$, including a classification of graded automorphisms of $B_q(t^d)$. We show that, in general, $B_q(t^d)^G$ is never AS regular. First, we define the following subset of $B_q(t^d)$:
\begin{align}\label{eq.T}
T_q:= \{ h \in \kk_q[u,v] : hu=quh \quad\text{and}\quad hv=q\inv vh \}.
\end{align}

\begin{lemma}\label{lem.Td}
Let $B=B_q(t^d)$ for some $d \geq 1$. 
\begin{enumerate}
\item If $q$ is a nonroot of unity, then $T_q=\{0\}$.
\item If $q$ is an primitive $\ordq$-root of unity, then 
\[ T_q=\mathrm{span}_\kk\{u^{b\ordq-1}v^{c\ordq-1}:b,c\in\NN\} \subset \kk_q[u,v].\]
\end{enumerate}
\end{lemma}
\begin{proof}
Assume $h\in T_q$. If $u^{m_1}v^{m_2}$ is a summand of $h$, then 
\begin{align*}
    (u^{m_1}v^{m_2})u &= q^{-m_2}u(u^{m_1}v^{m_2}) \\
    (u^{m_1}v^{m_2})v &= q^{m_1}v(u^{m_1}v^{m_2}).
\end{align*}
So, $q^{m_2+1}=1$ and $q^{m_1+1}=1$. Hence, if $q$ is a nonroot of unity, then $h=0$. Otherwise, $\ordq = \ord(q)$ must divide $m_1+1$ and $m_2+1$ for each $(m_1,m_2) \in \supp(h)$. Thus, in case (2), $T_q \subset \kk\{u^{b\ordq-1}v^{c\ordq-1}:b,c\in\NN\}$. The reverse inclusion is clear.
\end{proof}

\begin{lemma}
Fix $q \in \kk^\times$ and $\degf=\deg(f)$. 

(1) Let $a,\xi \in \kk^\times$ such that $\xi$ is a $(\degf+1)$-root of unity and for all $i \in \supp(f)$, $a^i=a^d$ and $\xi^i = \xi\inv$. Let $h \in T_q$. Then the following determines an automorphism of $B_q(f)$:
\begin{align}\label{eq.phi_aut}
	\phi_{a,\xi,h}(u)= au, \qquad 
    \phi_{a,\xi,h}(v)=\xi a v, \qquad 
    \phi_{a,\xi,h}(w)= \xi\inv a^{\degf-1} w + h.
\end{align}

(2) If $q=\pm 1$, then there is an automorphism $\tau$ of $B_q(f)$ given by 
\begin{align}\label{eq.tau}
\tau(u) = v, \qquad \tau(v) = u, \qquad \tau(w) = w.
\end{align}
\end{lemma}
\begin{proof}
We must show that $\phi$ respects the relations on $B_q(f)$. This is clear for the relation $uv-qvu$. Now
\begin{align*}
\phi(w)\phi(u)&-q\phi(u)\phi(w) - f(\phi(v)) \\
    &= \xi^\degf a^{\degf} (wu - quw) + a(hu-quh) - f(\xi a v) \\
    &= \xi^\degf a^{\degf} f(v) - \xi^\degf a^{\degf} f(v) = 0.
\end{align*}
The last relation is checked similarly. The proof for $\tau$ is straightforward.
\end{proof}

\begin{lemma}
(1) Let $\phi_{a_i,\xi_i,h_i}$ be as in \eqref{eq.phi_aut} for $i=1,2$. Then
\[ \phi_{a_2,\xi_2,h_2} \circ \phi_{a_1,\xi_1,h_1} = \phi_{a_1a_2,\xi_1\xi_2,h_3}\] 
where $h_3 = h_2 + \phi_{a_2,\xi_2,h_2}(h_1) \in T_q$.

(2) Suppose $q=\pm 1$. Let $\phi_{a,\xi,h}$ be as in \eqref{eq.phi_aut} and let $\tau$ be as in \eqref{eq.tau}. Then
\[ \tau \circ \phi_{a,\xi,h} = \phi_{\xi a,\xi\inv,\tau(h)} \circ \tau.\]
\end{lemma}
\begin{proof}
For (1), the relation is clear when evaluated on $u$ and $v$. When evaluated on $w$ we have
\begin{align*}
(\phi_{a_2,\xi_2,h_2} \circ \phi_{a_1,\xi_1,h_1})(w)
    &= \phi_{a_2,\xi_2,h_2}(\xi_1\inv a_1^{d-1} w + h_1) \\
    &= \xi_1\inv a_1^{d-1} \phi_{a_2,\xi_2,h_2} (w) + \phi_{a_2,\xi_2,h_2}(h_1) \\
    &= (\xi_1\xi_2)\inv (a_1a_2)^{d-1} w + (h_2 + \phi_{a_2,\xi_2,h_2}(h_1)) \\
    &= \phi_{a_1a_2,\xi_1\xi_2,h_3}(w).
\end{align*}
Since $\phi_{a_2,\xi_2,h_2}$ only scales the terms of $h_1$, then it is clear that $h_3 \in T_q$.

For (2) we have
\begin{align*}
(\tau \circ \phi_{a,\xi,h})(u) 
    &= av
    = \xi\inv (\xi a v)
    = \phi_{\xi a,\xi\inv,h}(v)
    = (\phi_{\xi a,\xi\inv,h} \circ \tau)(u) \\
(\tau \circ \phi_{a,\xi,h})(v) 
    &= \xi au
    = \phi_{\xi a,\xi\inv,h}(u)
    = (\phi_{\xi a,\xi\inv,h} \circ \tau)(v) \\
(\tau \circ \phi_{a,\xi,h})(w) 
    &= \xi\inv a^{d-1} w + \tau(h)
    = (\xi\inv)\inv (\xi a)^{d-1} w + \tau(h) \\
    &= (\phi_{\xi a,\xi\inv,\tau(h)} \circ \tau)(w).
\end{align*}
This proves the result.
\end{proof}

\subsection{Graded automorphisms}
We will assume throughout that $q \neq 1$. Fix ${\degf} \geq 2$, and let $g \in \Aut(B_q(t^{\degf}))$. For $b \in B_q(t^{\degf})$, write 
\[ g(b) = g_0(b) + g_1(b) + g_2(b) + \cdots \]
where $g_k(b)$ denotes the degree $k$ component of the image of $b$.

\begin{lemma}\label{lem.phi0}
Fix $q \neq 1$ and ${\degf} \geq 2$, and let $g \in \Aut(B_q(t^{\degf}))$. Then $\psi_0=0$.
\end{lemma}
\begin{proof}
Because of the relation $xy-qyx$ and $q \neq 1$, standard arguments show that $g_0(u)=g_0(v)=0$ (see, e.g. \cite[Lemma 2.2]{Giso1}). In a similar way we have
\[ 0 = g_1(wu-quw-v^\degf) = (1-q)g_0(w) g_1(u),\]
so $g_0(w)=0$. 
\end{proof}

\begin{proposition}\label{prop.graut}
Fix $q \neq 1$ and ${\degf} \geq 2$. Let $\phi \in \Autgr(B_q(t^{\degf}))$. If $q \neq -1$, then 
$\phi=\phi_{a,\xi,h}$ for some $a \in \kk^\times$, $\xi$ a $(\degf+1)$-root of unity, and $h \in T_q$

If $q=-1$ and $\psi \in \Autgr(B_{-1}(t^2))$, then either $\psi=\phi$ or $\psi=\phi \circ \tau$ where $\phi=\phi_{a,\xi,h}$ as above. 
\end{proposition}
\begin{proof}
Write
\[
	\phi(u) = a_1 u + b_1 v + c_1 w, \quad
	\phi(v) = a_2 u + b_2 v + c_2 w, \quad
	\phi(w) = h + c_3 w
\]
for $a_i,b_i,c_i \in \kk$ and $h \in \kk_q[u,v]$ with $\deg(h)=d-1$. If $\degf > 2$, then $c_1=c_2=0$.

Suppose $\degf=2$. The coefficient of $w^2$ in $\phi_2(uv-qvu)$ is $c_1c_2(1-q)$. Thus, $c_1=0$ or $c_2=0$. If $c_2=0$, then the coefficient of $w^2$ in $\phi_2(wv-q\inv vw-u^2)$ is $c_1^2=0$, so $c_1=0$. A similar argument holds if we assume first that $c_2=0$. Thus, we may assume $c_1=c_2=0$ regardless of $\degf$.

Now we have that $\phi$ restricts to an automorphism of $\kk_q[u,v]$. If $q \neq -1$, then $\phi_1(u) = a_1u$ and $\phi_2(v)=b_2v$ with $a_1,b_2\neq 0$. If $q=-1$, then either $\phi_1$ has the form already given or $\phi_1(u)=b_1v$ and $\phi_1(v)=a_2u$ with $a_2,b_1 \neq 0$. If $q=-1$ and $\phi_1$ is of the second form, then after composing with the automorphism $\tau$ given in \eqref{eq.tau}, we may assume that $\phi_1$ is of the first form. 

Now
\begin{align*}
\phi_2(wu-q uw-v^\degf)
    &= a_1(hu-quh) + (a_1c_3-b_2^\degf)v^\degf \\
\phi_2(wv-q\inv vw-u^\degf)
    &= b_2(hv-q\inv vh) + (b_2c_3-a_1^\degf)u^\degf.
\end{align*}
Hence, either $h=0$ or else $hu=quh$ and $hv=q\inv vh$, so $h \in T_q$. This leaves the conditions $c_3a_1 = b_2^\degf$ and $c_3b_2 = a_1^\degf$. Solving both for $c_3$ shows that $a_1^{\degf+1}=b_2^{\degf+1}$.
\end{proof}

By Lemma \ref{lem.Td}, the map $\phi_{a,\xi,h}$ is diagonal when $q$ is not a root of unity. The next result shows that the maps $\phi_{a,\xi,h}$ can often be diagonalized in the root-of-unity case.

For this result, we define $C_h$ to be the set of $(b,c)\in \NN^2$ such that $u^{bn-1}v^{cn-1}$ is a nonzero summand of $h$. Write 
\begin{align}\label{eq.hform}
    h = \sum_{(b,c) \in C_h} \alpha_{b,c} u^{bn-1} v^{cn-1}
\end{align}
where $\alpha_{b,c} \in \kk^\times$. 

\begin{lemma}\label{lem.diag}
Let $\phi=\phi_{a,\xi,h}$ and write $h$ as in \eqref{eq.hform}. Suppose for all $(b,c)\in C_h$, $\xi^{cn}a^{bn+cn-d-1} \neq 1$. Set 
\[ \beta_{b,c} = 1 - \xi^{cn}a^{bn+cn-d-1}\] 
and set
\[ \widehat{h} = \xi a^{1-d} \sum_{(b,c) \in C_h} \alpha_{b,c}\beta_{b,c}\inv u^{bn-1}v^{cn-1}.\]
Then $\widetilde{w} = w + \widehat{h}$ is an eigenvector for $\phi$.
\end{lemma}
\begin{proof}
We prove this for $h=u^{bn-1}v^{cn-1}$. The general case is entirely analogous. We have
\begin{align*}
\phi(\widetilde{w}) 
    &= \phi(w+\xi a^{1-d} \beta_{b,c}\inv h) \\
    &= (\xi\inv a^{\degf-1} w + h) + \xi a^{1-d} \beta_{b,c}\inv a^{(bn-1)+(cn-1)} \xi^{cn-1}
    h \\
    &= \xi\inv a^{\degf-1} \left( w +  (\xi a^{1-d} + \xi a^{1-d} \beta_{b,c}\inv a^{bn+cn-1-d} \xi^{cn} )
    h \right) \\
    &= \xi\inv a^{\degf-1} \left( w +  \xi a^{1-d}\beta_{b,c}\inv h \right),
\end{align*}
as claimed.
\end{proof}

The relations for the generating set $\{u,v,\widetilde{w}\}$ of $B_q(t^d)$ are identical to those in \eqref{eq.Bq}. In particular,
\[ \widetilde{w} u - qu\widetilde{w} - v^d 
    = \widetilde{w} v - q\inv v\widetilde{w} - u^d 
    = 0.\]

\begin{proposition}\label{prop.diag}
Let $\phi=\phi_{a,\xi,h}$ and write $h$ as in \eqref{eq.hform} above. Then $\phi$ has finite order if and only if $\ord(a)<\infty$ and $\xi^{cn}a^{bn+cn-d-1} \neq 1$ for all $(b,c)\in C_h$.
\end{proposition}
\begin{proof}
Suppose $\ord(a)<\infty$ and $\xi^{cn}a^{bn+cn-d-1} \neq 1$ for all $(b,c)\in\supp(h)$. 
Let $\{u,v,\widetilde{w}\}$ be as in Lemma \ref{lem.diag}. Then
$\phi$ acts diagonally on this new generating set:
\[
\phi(u) = au, \qquad
\phi(v) = \xi a v, \qquad
\phi(\widetilde{w}) = \xi\inv a^{d-1} \widetilde{w}.
\]
Consequently, $\ord(\phi)=\lcm(\ord(\xi),\ord(a))<\infty$.

Now suppose that there exists $(b,c) \in \supp(h)$ such that $\xi^{cn}a^{bn+cn-d-1} = 1$, thus $a^{bn+cn-d-1}=\xi^{-cn}$, and since $\xi$ is a $(d+1)$-root of unity, then $(a^{bn+cn-d-1})^{d+1}=1$. Hence, $a\in\kk$ is a root of unity.

Again, it will suffice to consider $h=u^{bn-1}v^{cn-1}$. Set $\gamma=\xi\inv a^{d-1}$. Then we have
\[ \phi(h) = \xi^{cn-1}a^{(bn-1)+(cn-1)}h = \gamma h.\]
We have $\phi(w)=\gamma w+h$ and an induction shows that for any $k \geq 1$,
\[ \phi^k(w) = \gamma^k w + k\gamma h.\]
This proves the claim.
\end{proof}

For the remainder, we assume $h=0$ and write the maps \eqref{eq.phi_aut} by $\phi_{a,\xi}$. We now consider the invariant theory of the $B_q(t^d)$ under cyclic groups generated by the $\phi_{a,\xi}$. First we review some background. For further reading the reader is directed to the survey article of Kirkman \cite{Ki}.

Let $A=\bigoplus_{k \in \NN} A_k$ be an $\NN$-graded algebra and $\psi$ a graded automorphism of $A$. The \emph{trace series} of $\psi$ is
\[ \Tr(\psi,t) = \sum_{k=0}^\infty \tr(\psi\restrict{A_k})t^k\]
where $\tr(\psi\restrict{A_k})$ denotes the usual trace of $\psi$ as a linear map acting on $A_k$. Now if $G$ is a finite group of graded automorphisms of $A$, then Molien's Theorem says that
\[ H_{A^G} = \frac{1}{|G|} \sum_{g \in G} \Tr_A(g,t).\]
The \emph{homological determinant} is a group homomorphism $\hdet:\Aut(A) \to \kk^\times$ defined originally in terms of local cohomology \cite{gourmet}. In the where $A$ is a polynomial ring, then $\hdet=\det$. We do not give a formal definition of $\hdet$ here but use a particular trick to compute it for the $\phi_{a,\xi}$.

\begin{lemma}\label{lem.graut_inv}
Fix $q \in \kk^\times$ and $d \geq 2$. Set $B=B_q(t^d)$. Let $a \in \kk^\times$ and $\xi$ be a $(\degf+1)$-root of unity. Set $\phi=\phi_{a,\xi}$.Then
\begin{enumerate}
\item $\ord(\phi)=\lcm(\ord(a),\ord(\xi))$,
\item $\displaystyle\Tr_B(\phi,t) = (1-at)\inv(1-\xi a t)\inv(1-\xi\inv a^{d-1}t^{d-1})\inv$,  and
\item $\hdet(\phi) = a^{d+1}$.
\end{enumerate}
\end{lemma}
\begin{proof}
(1) and (2) are clear because $\phi$ acts diagonally on an appropriate basis of $B$.

By (2), we have that $\Tr_B(\phi,t) = p(t)\inv$ where
\[  p(t) =
    1 - (1+\xi) at + \xi a^2 t^2 - \xi\inv a^{d-1} t^{d-1} + (\xi\inv + 1)a^d t^d - a^{d+1}t^{d+1}.
\]
Hence, as a Laurent series in $t\inv$,
\[
\Tr_B(\phi,t) = (-1) (a^{d+1})\inv t^{-(d+1)} + \text{(lower terms)}.
\]
Thus, by \cite[Lemma 2.6]{gourmet}, $\hdet(\phi) = a^{d+1}$.
\end{proof}

Suppose $A$ is AS regular with GK dimension $n$. A graded automorphism $g$ is a \emph{reflection} of $A$ if
\[ \Tr(g,t) = \frac{1}{(1-t)^{n-1} q(t)}\]
where $q(1) \neq 0$. Suppose further that $A$ is noetherian and $G$ a finite subgroup of graded automorphisms of $A$. By \cite[Theorem 2.4]{KKZ5}, if $A^G$ has finite global dimension, then $G$ contains a nontrivial reflection of $A$.

\begin{theorem}\label{thm.gorenstein}
Fix $q \neq \pm 1$ and $d \geq 2$. Set $B=B_q(t^d)$. 
\begin{enumerate}
    \item If $H$ is a nontrivial finite subgroup of $\Autgr(B)$, then $B^H$ is not AS regular.
    \item If $H$ is a finite subgroup of $\{ \phi_{a,\xi,h} : a^{d+1}=1\}$, 
    then $B^H$ is AS Gorenstein.
\end{enumerate}
\end{theorem}
\begin{proof}
(1) Let $\phi \in H$ be a reflection. By Propositions \ref{prop.graut} and \ref{prop.diag}, $\phi=\phi_{a,\xi}$ for some choice of parameters. Write
\[ p(t) = (1-at)(1-\xi a t)(1-\xi\inv a^{d-1}t^{d-1})\]
so that $\Tr(\phi,t)=p(t)\inv$.

Since $\phi$ is a reflection, then $(1-t)^2 \mid p(t)$. So, $1$ is a root of at least one of the factors.

If $1$ is a root of $(1-at)$, then $a=1$. But then $(1-t) \mid (1-\xi\inv t)(1-\xi\inv t^{d-1})$. Since $(1-t)$ divides one of the factors, this forces $\xi=1$, so $\phi=\id_B$.

Now suppose $a \neq 1$. So then $(1-t)$ divides $(1-\xi at)$ and $(1-\xi\inv a^{d-1} t^{d-1})$. Hence, $\xi a=1$ and $\xi\inv a^{d-1}=1$ or, equivalently, $\xi\inv = a$ and $a^d=1$. This implies that $a$ is a $d$-root of unity, but $\xi$ is a $(d+1)$-root of unity, so $a=\xi=1$.

We conclude that $H$ contains no nontrivial reflections. Since $B$ is noetherian, then the result follows from \cite[Theorem 2.4]{KKZ5}.

(2) By Lemma \ref{lem.graut_inv}, $\hdet(\phi)=1$ for all $\phi \in H$. Now $B^H$ is Gorenstein by \cite[Theorem 3.3]{gourmet}.
\end{proof}

\begin{remark}\label{rmk.inv}
In the case of $d=2$ (i.e., the quadratic case), there is an alternate way to approach (2) and (3) in Lemma \ref{lem.graut_inv}. Set $B=B_q(t^2)$ and $\phi=\phi_{a,\xi}$.

For (2), we first compute the Koszul dual of $B$ using standard techniques. Let $\ox,\oy,\oz$ be a dual basis for $B$. Then
\[
B^! = \kk\langle u,v,w \rangle/(\ou\ov+q\inv \ov\ou, \ow\ou+q\inv \ou\ow, \ow\ov+q\ov\ow, \ow\ou+\ov^2, \ow\ov+\ou^2, \ow^2 ).
\]
Now a $\kk$-basis for $B^!$ is
$\{ 1, \ou, \ov, \ow, \ou^2, \ou\ov, \ov^2, \ou^3\}$.
The graded automorphism $\phi$ of $B$ induces a graded automorphism $\phi^T$ on $B^!$. Since $\phi$ is diagonal, then $\phi(\ou)= a\ou$, $\phi(\ov)=\xi a \ov$, and $\phi(\ow)= \xi^2 a \ow$.
Hence,
\[ \Tr_{A^!}(\phi^T,t) = 1 + (a + \xi a + \xi^2 a) t
	+ (a^2 + \xi a^2 + \xi^2 a^2) t^2 + a^3 t^3
		= 1 + a^3 t^3.
\]
Then by \cite[Corollary 4.4]{Jiz},
\[ \Tr_{A}(\phi,t) = \Tr_{A^!}(\phi^T,-t)\inv = (1-a^3 t^3)\inv.\]

For (3), let $\bw$ be the superpotential corresponding to $B$ from \eqref{eq.potential}. By \cite[Theorem 1.2]{MS}, $\phi(\bw)=(\hdet(\phi_{a,\xi}))\bw = a^3 \bq$.
\end{remark}

\section{Further questions and comments}
\label{sec.questions}

We end with some additional questions based on the above discussion. We hope that this work encourages further study of algebras with trivial ozone groups. In \cite{CGWZ3}, the two trivial ozone families were the $B_q(t^2)$ for $q$ a root of unity and (most) PI Sklyanin algebras. Just as the $B_q(f)$ generalize the first family, the following algebras generalize the Sklyanin algebras.

\begin{definition}
For generic $a,b \in \kk$ and $f \in \kk[t]$, define
\[ 
S_{a,b}(f) 
    = \kk\langle u,v,w \mid
        auv+bvu+f(w), avw+bwv+f(u), awu+buw+f(v)
    \rangle.
\]
\end{definition}

One recovers the classical Skylanin algebras when $f=t^2$, as well as various PBW deformation when $\deg f \leq 2$. We expect that similar results hold here but arguments will require a more geometric approach.

\begin{question}
For what parameters $a,b,f$ does $S_{a,b}(f)$ have trivial ozone group?
\end{question}

In Section \ref{sec.center} we gave partial information on the center of the $B_q(f)$ at roots of unity $q$. In general this problem seems rather subtle.

\begin{question}
Let $q$ be a nontrivial root of unity. For general $f$, what is the center of $B_q(f)$?
\end{question}

In Section \ref{sec.invariant} we studied the invariant theory of the automorphisms $\phi_{a,\xi,h}$. In the case that $q=\pm 1$, the automorphism group contains reflections. In particular, one can show that $\tau$ itself is a reflection and $B_{\pm 1}(t^d)^{\langle \tau \rangle}$ is AS regular for $d \geq 2$. In general, there are additional automorphisms $\psi_{a,\xi,h} = \phi_{a,\xi,h} \circ \tau$ given by
\[
\psi_{a,\xi,h}(u) = \xi a v, \quad
\psi_{a,\xi,h}(v) = a u, \quad
\psi_{a,\xi,h}(w) = \xi\inv a^{d-1} w + h,
\]
where $a \in \kk^\times$, $\xi$ a $(d+1)$-root of unity, and $h \in T_q$. A similar strategy as that employed in Proposition \ref{prop.graut} applies here and one can diagonalize these maps. However, analyzing the possible invariant rings is significantly more challenging.

\begin{question}
Under what conditions is $B_{\pm 1}(t^d)^G$ AS regular?
\end{question}

Let $A=\kk[x_1,x_2,x_3]$ and $f \in \kk[t]$. Define a Poisson bracket on $A$ by
\[ 
\{ u,v \} = uv, \qquad
\{ w,u \} = uw + f(v), \qquad
\{ v,w \} = vw + f(u).
\]
Denote these Poisson algebras by $A(f)$. The $A(f)$ are unimodular in general and $A(t^2)$ is the semiclassical limit of the $B_q(t^2)$.

In \cite{CGWZ4}, the authors studied ozone groups of Poisson algebras in characteristic $p>0$. As these are Poisson algebras, ozone automorphisms are replaced by certain Poisson derivations (called \emph{log-ozone derivations}). By \cite[Theorem 4.7]{CGWZ4}, the log-ozone group of $A(t^2)$ is trivial.

\begin{question}
For general $f \in \kk[t]$, is the log-ozone group of $A(f)$ trivial?
\end{question}

\end{document}